\documentclass[12pt, twoside]{amsart}
\usepackage{a4, amssymb, amsmath, amsthm, latexsym, cite, hyperref}
\usepackage{graphicx, amsfonts, ulem, float}
\usepackage{tikz,tikz-network,pgfplots} 
\usetikzlibrary{shapes, arrows, positioning}
\usepackage[font=small, labelsep=none, figurename=Fig.]{caption}

\newtheorem{theorem}{Theorem}[section]

\newtheorem{definition}[theorem]{Definition}
\newtheorem{example}[theorem]{Example}
\newtheorem{lemma} [theorem]{Lemma}
\newtheorem{notation}[theorem]{Notation}

\newtheorem{remark}[theorem]{Remark}
\vspace{5cm}
\begin{document}
\title{\bf Adjacency Spectra of Semigraphs}
\address{Dept. of Maths, College of Engineering Pune, Maharashtra-411005, India.}
\author{Pralhad M. Shinde }
\email{pralhadmohanshinde@gmail.com}
\date{}
\maketitle

\thispagestyle{empty}

\begin{abstract}	
{\footnotesize  In this paper, we define the adjacency matrix of a semigraph. We give the conditions for a matrix to be semigraphical and give an algorithm to construct a semigraph from the semigraphical matrices. We derive lower and upper bounds for largest eigenvalues.  We study the eigenvalues of adjacency matrix of two types of star semigraphs. 
}
 \end{abstract}

{\small \textbf{Keywords:} {\footnotesize Adjacency matrix of semigraph, Excess adjacency matrix of semigraph } }

{\small \textbf{AMS classification:}{\footnotesize \;\;05C50; 15A18}}

\vskip 1cm
\section{Introduction}
Graph theory is an old and very evolved subject having abundant number of applications to industry as well as other subjects. There are several generalizations of graph structure available in literature. Hypergraph ~\cite{berge} is one of the most studied generalization of graphs and several graph theory results have been studied for hypergraphs. Sampathkumar ~\cite{smt} defined the semigraph as a generalization of graph which can be thought as a linear hypergraph with ordered edges. The ordering in vertices of edges makes semigraphs different from hypergraphs. The most sought after tool to study graph properties is linear algebra by associating a matrix to the graph and so is for semigraphs. In ~\cite{cmd}, adjacency matrix for semigraphs is studied but their adjacency matrix does not enjoy the symmetric property and hence loses the advantage of rich theory of symmetric matrices. In this paper, we define the adjacency matrix for a semigraph in a natural way so that it not only matches with the adjacency matrix of graph when the semigraph is a graph but also enjoys the symmetric property. Hence, it would turn out to be a powerful tool to study semigraphs. 

This paper is organized as follows. In Section 2, we define the adjacency matrix of semigraphs and study some basic properties.  In section 3, we obtain expressions for degree sum and eigenvalues square sum. Further, we give lower and upper bounds to the largest eigenvalue of adjacency matrix of a semigraph. In Section 4,  we analyze the eigenvalues of two types of star semigraphs.

\section*{Preliminaries}
For all basic definitions and standard notations please refer to \cite{smt}.\\
We recall some definitions here;
\begin{definition} \label{def:1}
Let $V$ be a non-empty set having $n$ elements. A $semigraph$ is a pair G=$(V, E),$ where the elements of $V$are called vertices and $E$ is a set of ordered $k$-tuples of distinct vertices, whose elements are called edges of $G$; for $n\geq 2,$ satisfying the following conditions:

\begin{enumerate}
	\item  Any two edges have at most one vertex in common
	\item Two edges $(u_1,u_2,\cdots,u_k)$ and $(v_1,v_2,\cdots,v_r)$ are considered to be equal if
	\begin{enumerate}
	\item r = k and
	\item either $u_i = v_i$ for $1\leq i \leq k,$  or $u_i=v_{k-i+1}$, for $1\leq i \leq k.$ 
	\end{enumerate}
\end{enumerate}
\end{definition}

Thus the edge $(u_1,u_2,\cdots,u_r)$ is same as
$(u_r,u_{r-1},\cdots,u_1)$.\\\\
Two vertices $v_i,\; v_j$ in a semigraph are said to be {\bf\itshape adjacent} if they belong to the same edge  and are said to be {\bf\itshape consecutively adjacent} if in addition they are consecutive in order as well. 
\vskip2mm
\noindent For the edge $e=(u_1,u_2,\cdots,u_n)$, $u_1$ and $u_n$ are called the {\bf \it end} vertices of $e$ and $u_2,u_3,\cdots,u_{n-1}$ are called the {\bf \it middle} vertices of $e$. Note that $u_i, \; u_j$ are adjacent for all $1\leq i, j\leq n$ while $u_i, u_{i+1}$ are consecutively adjacent for all $1\leq i\leq n-1$. \\

For a semigraph, we define following types of vertices and edges:
\begin{enumerate}
\item $u_i$ is said to be a pure end vertex if it is an end vertex of every edge to which it belongs. 
\item $u_i$ is said to be a pure middle vertex if it is a middle vertex of every edge to which it belongs. 
\item $u_i$ is said to be a middle end vertex if it is middle vertex of at least one edge and end vertex of at least one other edge.
\item An edge $e=(u_1, u_2, \cdots, u_k),\; k\geq2$ is said to be full edge if $u_1$ and $u_k$ are pure end vertices.
\item An edge $e=(u_1, u_2, \cdots, u_k),\; k>2$ is said to be an half edge if either $u_1$ or $u_k$ (or both) are middle end vertices.
\item An edge $e=(u_1, u_2)$ is said to be a quarter edge if both  $u_1$ and $u_2$ are middle end vertices while $e=(u_1, u_2)$ will be  half edge if exactly one of $u_1$ and $u_2$ is a middle end vertex and other is a pure end vertex. 
\end{enumerate}
 \par For a full edge $e=(u_1, u_2, \cdots, u_k)$, $(u_i, u_{i+1})\; \forall\; 1\leq i\leq k-1$ is called a partial edge of $e$ while for a half edge $e=(u_1, u_2, \cdots, u_{k-1}, u_k)$, $(u_1, u_2)$ is called partial half edge of $e$ if $u_1$ is middle end vertex,  $(u_{k-1}, u_k)$ is a partial half edge of $e$ if $u_k$ is middle end vertex and $(u_i, u_{i+1})\; \forall\; 2\leq i\leq k-2$ are partial edges of $e$. Thus, any half edge can have at most two partial half edges. Also an half edge is a partial half edge iff it contains only two vertices and half edge with partial half edges must have at least 3 vertices.  
\vskip2mm
\begin{example}
Let $G = (V, E)$ be a semigraph, with $V = \{v_1,v_2,\cdots,v_{10}\}$ as a vertex set and 
$E = \{(v_1,v_2,v_3,v_4,v_5),$ $(v_1,v_7,v_8),$ $(v_2,v_6,v_8),$ $(v_1,v_9),$ $(v_6,v_7)\}.$ as an edge set.
\end{example}

\begin{figure}[h]
\centering
 \begin{tikzpicture}[yscale=0.5]
  \Vertex[x=-1, y=2, size=0.2,  label=$v_9$, position=left, color=black]{I} 
 \Vertex[size=0.2,  label=$v_1$, position=below, color=black]{A} 
  \Vertex[x=3, size=0.2,label=$v_2$, color=none, position=below]{B} 
  \Vertex[x=5,size=0.2,label=$v_3$,position=below, color=none]{C} 
   \Vertex[x=7, size=0.2, label=$v_4$,position=below, color=none]{D}
   \Vertex[x=9, size=0.2, label=$v_5$,position=below, color=black]{D}
  \Vertex[x=3,y=2,size=0.2,label=$v_6$,position=right, color=none]{E}  
  \Vertex[x=3, y=4, size=0.2, label=$v_8$,position=above, color=black]{F}
  \Vertex[x=1.5,y=2,size=0.2,label=$v_7$,position=left, color=none]{G}  
    \Vertex[x=5,y=3,size=0.2,label=$v_{10}$,position=right, color=black]{H}  
  \Edge(A)(B) \Edge(B)(C) \Edge(C)(D) \Edge(A)(G) \Edge(G)(F) \Edge(F)(E)  \Edge(I)(A)
  \draw[thick](2.9,0.3)--(3.1, 0.3);
  \draw[thick](1.65,2)--(2.85, 2);
  \draw[thick](3, 1.85)--(3, 0.3);
  \draw[thick](2.85,1.75)--(2.85, 2.2);
 \draw[thick](1.65,1.75)--(1.65, 2.2);
\end{tikzpicture}
\caption{}
\label{fig:1}
 \end{figure}

In Fig.~\ref{fig:1}, vertices $v_1,v_5,v_8,$and$v_9$ are the pure end vertices; $v_3, v_4$ are pure middle vertices; $v_2$,$v_6,$and$v_7$ are the middle end vertices and $v_{10}$ is an isolated vertex. Further, $(v_1, v_9)$, $(v_1, v_7, v_8)$, $(v_1, v_2, v_3, v_4, v_5)$ are full edges whereas $(v_2, v_6, v_8)$ is an half edge with only $(v_2, v_6)$ as a partial half edge. Note that $(v_6, v_7)$ is the a quarter edge. 
 \begin{definition}  \label{def:2}
 A semigraph $G=(V, E)$ is said to be connected if for any two vertices $u,\; v\;\in E$, there exist a sequence of edges $e_{i_1},\cdots, e_{i_p}$ for some $p$ such that $u\in e_{i_1},\; v\in e_{i_p}$ and $|e_{i_{j}}\cap e_{i_{j+1}}|=1,\; \forall\; 1\leq j\leq p-1$.
 \end{definition}

\begin{notation}
Throughout this paper, we assume that semigraph is connected and $G=(V, E)$ denotes the semigraph with $n$ vertices and $m$ edges such that
\begin{itemize}
\item $m_1$ is the number of full edges 
\item $m_2$ is the number quarter edges
\item $m_3$ is the number of half edges with one partial half edge
\item $m_4$ is the number of half edges with two partial half edges  
\end{itemize}  
Note that $m=m_1+m_2+m_3+m_4$ and if $G$ is a graph then $m_2=m_3=m_4=0$ and $m=m_1$.
\end{notation}
\vskip2mm

 \begin{definition} {\textbf{Graph Skeleton of a semigraph}} \label{def:3}\\
Let $G=(V, E)$ be a semigraph. The graph skeleton of $G$ is an underlined graph structure $G^{S}$ of the semigraph on $V$, where two vertices $v_{i}$, $v_{j}$ are adjacent in $G^{S}$ iff $v_{i}$ and $v_{j}$ are consecutively adjacent in $G.$ We write $v_i \sim_{S} v_j$.  
 \end{definition}
 \begin{example} Consider the following semigraph
 \begin{figure}[h]
 \centering
 \begin{tikzpicture}[yscale=0.5]
 \Vertex[size=0.2,  label=$v_1$, position=below, color=black]{A} 
  \Vertex[x=3, size=0.2,label=$v_2$, color=none, position=below]{B} 
  \Vertex[x=5,size=0.2,label=$v_3$,position=below, color=none]{C} 
   \Vertex[x=7, size=0.2, label=$v_4$,position=below, color=none]{D}
  \Vertex[x=3,y=2,size=0.2,label=$v_7$,position=right, color=none]{E}  
  \Vertex[x=3, y=4, size=0.2, label=$v_8$,position=above, color=black]{F}
  \Vertex[x=1.5,y=2,size=0.2,label=$v_6$,position=100, color=none]{G}  
    \Vertex[x=5,y=3,size=0.2,label=$v_9$,position=right, color=black]{H}  
        \Vertex[x=9,size=0.2,label=$v_5$,position=below, color=black]{I}  
  \Edge(A)(B)  \Edge(B)(C)
   \Edge(C)(D) \Edge(A)(G) \Edge(G)(F) \Edge(F)(E) \Edge(D)(I)  
    \draw[thick](2.9,0.3)--(3.1, 0.3);
  \draw[thick](1.65,2)--(2.85, 2);
  \draw[thick](3, 1.85)--(3, 0.3);
  \draw[thick](2.85,1.75)--(2.85, 2.2);
  \draw[thick](1.65,1.75)--(1.65, 2.2);
  \draw[thick](4.9,0.3)--(5.1, 0.3);
   \draw[thick](5, 2.85)--(5, 0.3);
 \end{tikzpicture}
 \caption{}
   \label{fig:2}
 \end{figure}

 The graph skeleton of the semigraph in Fig.~\ref{fig:2} is graph in Fig.~\ref{fig:3}
 \begin{figure}[h]
 \centering
 \begin{tikzpicture}[yscale=0.5]
 \Vertex[size=0.2,  label=$v_1$, position=below, color=black]{A} 
  \Vertex[x=3, size=0.2,label=$v_2$, color=black, position=below]{B} 
  \Vertex[x=5,size=0.2,label=$v_3$,position=below, color=black]{C} 
   \Vertex[x=7, size=0.2, label=$v_4$,position=below, color=black]{D}
  \Vertex[x=3,y=2,size=0.2,label=$v_7$,position=right, color=black]{E}  
  \Vertex[x=3, y=4, size=0.2, label=$v_8$,position=above, color=black]{F}
  \Vertex[x=1.5,y=2,size=0.2,label=$v_6$,position=100, color=black]{G}  
    \Vertex[x=5,y=3,size=0.2,label=$v_9$,position=right, color=black]{H}  
        \Vertex[x=9,size=0.2,label=$v_5$,position=below, color=black]{I}  
  \Edge(A)(B)  \Edge(B)(C)
   \Edge(C)(D) \Edge(A)(G) \Edge(G)(F) \Edge(F)(E) \Edge(D)(I)  \Edge(E)(G) \Edge(B)(E) \Edge(C)(H)
  \end{tikzpicture}
 \caption{}
 \label{fig:3}
 \end{figure}
 \end{example}
 
 Note that given a semigraph, it determines a unique graph skeleton but not conversely. \\ For any edge $e=(u_1, u_2,\cdots, u_k)\in E$, the graph skeleton of $e$ is a graph path of length $k-1$. For example, graph skeleton of edge $e=(v_1,v_2,v_3,v_3,v_5)$ of the semigraph Fig.~\ref{fig:2} is graph path Fig.~\ref{fig:4}
 \begin{figure}[h]
 \centering
  \begin{tikzpicture}[yscale=0.5]
   
 \Vertex[size=0.2, x=1, label=$v_1$, position=below, color=black]{A} 
 \Vertex[size=0.2, x=2, label=$v_2$, position=below, color=black]{B} 
 \Vertex[size=0.2, x=3, label=$v_3$, position=below, color=black]{C} 
 \Vertex[size=0.2, x=4, label=$v_4$, position=below, color=black]{D} 
 \Vertex[size=0.2, x=5, label=$v_5$, position=below, color=black]{E} 
 \Edge(A)(B) \Edge(B)(C) \Edge(C)(D) \Edge(D)(E)  
 \end{tikzpicture}
 \caption{}
 \label{fig:4}
 \end{figure}

 \section{Adjacency matrix}
  Adjacency matrix is a very powerful tool to study graph properties. To study semigraph properties along the same lines we define the adjacency matrix of a semigraph and study its properties. \\
  Let $G$=$(V, E)$ be a semigraph, with $V=\{v_1, v_2,\cdots, v_n\}$ as a vertex set and $E=\{e_1, e_2,\cdots, e_m\}$ as an edge set. Let $u_i, u_j \in e=(u_1, u_2,\cdots, u_k)$ for some $e \in E$. Let $d_{e}(u_i,u_j)$ denote the distance between $u_i$ and $u_j$ in the graph skeleton of $e$. The distance $d_{e}(u_i,u_j)$ is well-defined as each pair of vertices in semigraph belongs to at most one edge. 
  
  \begin{definition}  \label{def:4}
  We index the rows and columns of a matrix $A=(a_{ij})_{n\times n}$ by vertices $ v_1, v_2,\cdots, v_n,$ where $a_{ij}$ is given as follows:
$$a_{ij}=\begin{cases}
 d_{e}(v_i,v_j),& \text{if $v_i,\; v_{j}$ belong to a full edge or a half edge such that } \\
 &\text{$(v_i, v_j)$ is neither a partial half edge nor a quarter edge}\\
\;\;\; \frac{1}{2}, & \text{if $(v_i,\;  v_{j})$ is a partial half edge}\\
\;\;\;\frac{1}{4}, &\text{if $(v_i,  v_{j})$ is a quarter edge}\\
\;\;\;0,&\text{otherwise}
\end{cases}$$ 
\end{definition}
The matrix $A=(a_{ij})_{n\times n}$ is called the adjacency matrix of semigraph $G.$

\begin{example}
Consider the semigraph in Fig.~\ref{fig:5}. 
\begin{figure}[h]
\centering
 \begin{tikzpicture}[yscale=0.5]
 \Vertex[size=0.2,  label=$v_1$, position=below, color=black]{A} 
  \Vertex[x=3, size=0.2,label=$v_2$, color=none, position=below]{B} 
  \Vertex[x=5,size=0.2,label=$v_3$,position=below, color=none]{C} 
   \Vertex[x=7, size=0.2, label=$v_4$,position=below, color=none]{D}
  \Vertex[x=3,y=2,size=0.2,label=$v_7$,position=right, color=none]{E}  
  \Vertex[x=3, y=4, size=0.2, label=$v_8$,position=above, color=black]{F}
  \Vertex[x=1.5,y=2,size=0.2,label=$v_6$,position=100, color=none]{G}  
    \Vertex[x=5,y=3,size=0.2,label=$v_9$,position=right, color=black]{H}  
        \Vertex[x=9,size=0.2,label=$v_5$,position=below, color=black]{I}  
  \Edge(A)(B)  \Edge(B)(C)
   \Edge(C)(D) \Edge(A)(G) \Edge(G)(F) \Edge(F)(E) \Edge(D)(I)  
    \draw[thick](2.9,0.3)--(3.1, 0.3);
  \draw[thick](1.65,2)--(2.85, 2);
  \draw[thick](3, 1.85)--(3, 0.3);
  \draw[thick](2.85,1.75)--(2.85, 2.2);
  \draw[thick](1.65,1.75)--(1.65, 2.2);
  \draw[thick](4.9,0.3)--(5.1, 0.3);
   \draw[thick](5, 2.85)--(5, 0.3);
 \end{tikzpicture}
 \caption{}
 \label{fig:5}
 \end{figure}
The adjacency matrix $A$ of $G$ is given by 
 \[
     \bordermatrix{ & {v_1} & {v_2} & {v_3}& {v_4} & {v_5} & {v_6}& {v_7} & {v_8} & {v_9} \cr
       v_1 & 0&1&2&3&4&1&0&2&0 \cr
       v_2 & 1&0&1&2&3&0&\frac{1}{2}&2&0\cr
       v_3 & 2&1&0&1&2&0&0&0&\frac{1}{2} \cr
        v_4 & 3&2&1&0&1&0&0&0&0 \cr
       v_5 & 4&3&2&1&0&0&0&0&0 \cr
       v_6 &1&0&0&0&0&0&\frac{1}{4}&1&0 \cr
        v_7 &0&\frac{1}{2}&0&0&0&\frac{1}{4}&0&1&0\cr
       v_8 & 2&2&0&0&0&1&1&0&0\cr
       v_9 &0&0&\frac{1}{2}&0&0&0&0&0&0} \qquad
 \]
  \end{example}
\begin{remark}  
Note that each pair of vertices belongs to at most one edge and hence $a_{ij}$ in the adjacency matrix is well-defined. Adjacency matrix is symmetric and if the semigraph is a graph then above adjacency matrix coincides with the adjacency matrix of graph. 
\end{remark}
 Let $A^{S}$ be the adjacency matrix of graph skeleton $G^{S}$ of a semigraph $G$. We define an \textbf{Excess adjacency} matrix of Semigraph as follows

 \begin{definition}  \label{def:5}
 Let $G$ be a Semigraph with adjacency matrix $A$, let $A^S$ be the adjacency matrix of graph skeleton of $G$. The matrix $A^{E}$defined as $A^E=A-A^S$ is called Excess adjacency matrix of $G$. Thus, $A=A^{S}+A^{E}$.
 \end{definition}
Note that a semigraph is a graph if and only if $A_E$ is a zero matrix.
\begin{example}
Consider the Semigraph 
\begin{figure}[h]
\centering
 \begin{tikzpicture}[yscale=0.5]
 \Vertex[size=0.2,  label=$v_1$, position=below, color=black]{A} 
  \Vertex[x=3, size=0.2,label=$v_2$, color=none, position=below]{B} 
  \Vertex[x=5,size=0.2,label=$v_3$,position=below, color=black]{C} 
  \Vertex[x=3,y=2,size=0.2,label=$v_6$,position=right, color=none]{E}  
  \Vertex[x=3, y=4, size=0.2, label=$v_5$,position=above, color=black]{F}
  \Vertex[x=1.5,y=2,size=0.2,label=$v_4$,position=100, color=none]{G}  
  \Edge(A)(B) \Edge(B)(C) \Edge(A)(G) \Edge(G)(F) \Edge(E)(F)
  \draw[thick](1.7, 2)--(2.85,2);
  \draw[thick](3, 1.85)--(3, 0.3);
  \draw[thick](2.85, 0.3)--(3.15, 0.3);
 \draw[thick](2.85, 1.8)--(2.85, 2.2);
  \draw[thick](1.7, 1.8)--(1.7, 2.2);
    \end{tikzpicture}
   \caption{}
   \label{fig:6}
 \end{figure}
\begin{align*}
A=&A^S +A^E \\
\begin{pmatrix} 
0&1&2&1&2&0\\
1&0&1&0&2&\frac{1}{2}\\
2&1&0&0&0&0\\
1&0&0&0&1&\frac{1}{4} \\
2&2&0&1&0&1\\
0&\frac{1}{2}&0&\frac{1}{4}&1&0
\end{pmatrix}     
=&\begin{pmatrix} 
0&1&0&1&0&0\\
1&0&1&0&0&1\\
0&1&0&0&0&0\\
1&0&0&0&1&1 \\
0&0&0&1&0&1\\
0&1&0&1&1&0
\end{pmatrix}+\begin{pmatrix} 
0&0&2&0&2&0\\
0&0&0&0&2&-\frac{1}{2}\\
2&0&0&0&0&0\\
0&0&0&0&0&-\frac{3}{4} \\
2&2&0&0&0&0\\
0&-\frac{1}{2}&0&-\frac{3}{4}&0&0
\end{pmatrix}
\end{align*}
The matrix $A_E$ is an excess adjacency matrix of the semigraph in Fig.~\ref{fig:6}. It would be interesting to study the properties of excess matrix of a semigraph. 
\end{example}  
 We define the degree of vertices in semigraph as sum of graph degree $d^{S}$ and excess degree $d^{E}$. Let $A^{S}_{i},$ $A^{E}_{i}$ are $i^{th}$ rows of $A^{S},\;A^{E}$ respectively. 
\begin{definition}  \label{def:6}
Let $v_{i}$ be a vertex of $G$, define degree of vertex $v_i$ by\\ $d(v_{i})=d^{S}(v_i)+d^{E}(v_i)$, where $d^{S}(v_i)=A^{S}_{i}\mathbf{1}$ and $d^{E}(v_i)=A^{E}_{i}\mathbf{1}$, $\mathbf{1}$ being a column matrix with all entries 1.
\end{definition}
The degree of vertex $v_{2}$ in the Fig.~\ref{fig:6} is 4.5, where graph degree $d^{S}=3$ and excess degree $d^{E}=1.5$. \\
Graph degree is always an integer; excess degree is a number of the form $n+\alpha$, where $n$ is an integer and $\alpha\in \{0, 0.25, 0.5, 0.75\}.$ \\
\begin{remark}
Why do fractional entries of adjacency matrix makes sense?\\
First of all, we want a natural generalization of graph adjacency as every graph is an example of semigraph. Secondly, to respect the ordering of vertices, and to differentiate between middle end vs. pure end vertices of an edge we need to treat them differently. If we treat these vertices normally and define $a_{ij}=1$ if $v_i$ and $v_j$are consecutively adjacent then we get the same adjacency matrix for two non-isomorphic semigraphs. In other words, given the matrix we cannot determine the semigraph uniquely. 

\begin{example} Consider the following two semigraphs
\begin{figure}[h]
\centering
 \begin{tikzpicture}[yscale=0.5]
  \Vertex[size=0.2, x=-1, label=$v_1$, position=below, color=black]{G} 
 \Vertex[size=0.2, x=0,  label=$v_2$, position=below, color=none]{A} 
  \Vertex[x=1, size=0.2,label=$v_3$, color=none, position=-100]{B} 
  \Vertex[x=2,size=0.2,label=$v_4$,position=-85, color=none]{C} 
  \Vertex[x=3,size=0.2,label=$v_5$,position=below, color=black]{D} 
  \Vertex[x=1,y=2,size=0.2,label=$v_6$,position=above, color=black]{E}  
  \Vertex[x=1,y=-2,size=0.2,label=$v_7$,position=below, color=black]{F} 
   \Edge(A)(B) \Edge(B)(C) \Edge(C)(D) \Edge(F)(B) \Edge(E)(B) \Edge(G)(A)
  \draw[thick](1, 2)--(1.9, 0.3);
  \draw[thick](1,-2)--(1.9, -0.3);
  \draw[thick](1.85, 0.1)--(2, 0.5);
   \draw[thick](1.85, -0.1)--(2, -0.5);
 
  \Vertex[size=0.2, x=5,  label=$u_1$, position=below, color=black]{G'} 
  \Vertex[x=6,size=0.2,  label=$u_2$, position=below, color=none]{A'} 
  \Vertex[x=7, size=0.2,label=$u_3$, color=none, position=-100]{B'} 
  \Vertex[x=8,size=0.2,label=$u_4$,position=-85, color=none]{C'} 
  \Vertex[x=9,size=0.2,label=$u_5$,position=below, color=black]{D'}  
  \Vertex[x=8,y=2, size=0.2,label=$u_6$,position=above, color=black]{F'}
  \Vertex[x=8, y=-2, size=0.2,label=$u_7$,position=below, color=black]{E'}
   \Edge(A')(B') \Edge(B')(C') \Edge(C')(D') \Edge(F')(C') \Edge(E')(C') \Edge(G')(A')
   \draw[thick](7.1,0.3)--(8, 2);
   \draw[thick](7.1,-0.3)--(8, -2); 
   \draw[thick](7, 0.4)--(7.2,0.2);
  \draw[thick](7, -0.4)--(7.2,-0.2);

   \end{tikzpicture}
   \caption{}
   \label{fig:7}
 \end{figure}
 
 These two non-isomorphic semigraphs have the same adjacency matrix as follows.
 \begin{center}
$ \begin{pmatrix} 0&1&2&3&4&0&0\\
        1&0&1&2&3&0&0\\
        2&1&0&1&2&1&1 \\
        3&2&1&0&1&1&1 \\
        4&3&2&1&0&0&0 \\
        0&0&1&1&0&0&2 \\
         0&0&1&1&0&2&0
     \end{pmatrix}$
\end{center}
\end{example}
 Thus, if two vertices are consecutively adjacent and one of them or both are middle end vertices then assigning less weight to an edge makes more sense as it helps us to distinguish between consecutively adjacency where both vertices are pure end vertices. Further, it helps us to achieve uniqueness, see  theorem \ref{thm:1}. We get two different adjacency matrices for the above semigraphs with respect to our definition:
 
 \[
     \bordermatrix{ & {v_1} & {v_2} & {v_3}& {v_4} & {v_5} & {v_6}& {v_7}  \cr
       v_1 & 0&1&2&3&4&0&0\cr
       v_2 & 1&0&1&2&3&0&0\cr
       v_3 & 2&1&0&1&2&1&1 \cr
        v_4 & 3&2&1&0&1&\frac{1}{2}&\frac{1}{2} \cr
       v_5 & 4&3&2&1&0&0&0 \cr
       v_6 &0&0&1&\frac{1}{2}&0&0&2 \cr
        v_7 &0&0&1&\frac{1}{2}&0&2&0}, \qquad
     \bordermatrix{ & {u_1} & {u_2} & {u_3}& {u_4} & {u_5} & {u_6}& {u_7}  \cr
       u_1 & 0&1&2&3&4&0&0\cr
       u_2 & 1&0&1&2&3&0&0\cr
       u_3 & 2&1&0&1&2&\frac{1}{2}&\frac{1}{2} \cr
        u_4 & 3&2&1&0&1&1&1 \cr
       u_5 & 4&3&2&1&0&0&0 \cr
       u_6 &0&0&\frac{1}{2}&1&0&0&2 \cr
        u_7 &0&0&\frac{1}{2}&1&0&2&0} \qquad
 \]
 \end{remark}
 
 \subsection{Properties of Adjacency matrix}
 Some basic observations: \\
 Let $G$ be a semigraph with the adjacency matrix $A=(a_{ij})_{n\times n}$, we note the following properties
 \begin{enumerate}
  \item \label{prop:1} For every edge $e=(u_1,u_2, \cdots, u_r)$ with $r\geq 2$ which is not a quarter edge, the submatrix $A_e$ of the size $r\times r$ formed by $u^{th}_1,u^{th}_2, \cdots, u^{th}_r$ rows and columns is given by 
 \small{
 \[
     \bordermatrix{ & {u_1} & {u_2} & {u_3}& {u_4} & \cdots &u_{r-2} &u_{r-1}&{u_{r}} \cr
       u_1 & 0&\mu_1&2&3&\cdots&r-3&r-2&r-1\cr
       u_2 & \mu_1&0&1&2&\cdots&r-4&r-3&r-2\cr
       u_3 & 2&1&0&1&\cdots&r-5&r-4&r-3 \cr
        u_4 & 3&2&1&0&\cdots&r-6&r-5&r-4 \cr
    \vdots & \vdots&\vdots&\vdots&\vdots&\vdots&\vdots&\vdots&\vdots\cr
       u_{r-2} &r-3&r-4&r-5&r-6&\cdots&0&1&2\cr
       u_{r-1} &r-2&r-3&r-4&r-5&\cdots&1&0&\mu_2\cr
       u_{r} &r-1&r-2&r-3&r-4&\cdots&2&\mu_2&0} \qquad
 \]
 }
 where,  for $i\in \{1, 2\}$
 $$\mu_{i}=\begin{cases}
 1,& \text{if $u_i$ is pure end vertex}\\
 \frac{1}{2} ,&\text{if $u_i$ is middle end vertex}\\
\end{cases}$$ 
\begin{center}
 or
 \end{center}
For a quarter edge $e=(u_{i}, u_{j})$, the submatrix $A_e$ is
  \[
     \bordermatrix{ & {u_i} & {u_j}  \cr
       u_i & 0&\frac{1}{4}\cr
       u_{j} &\frac{1}{4}&0} \qquad
 \]
  We call the matrix $A_e$ as an edge sub-matrix of the semigraph.
\item  \label{prop:2} $a_{ij}=\frac{1}{4}$ iff $(v_i, v_j)$ is a quarter edge.
 \item  \label{prop:3} The number of quarter edges is equal to half the number of occurrences of $\frac{1}{4}$ in the matrix.
 \item  \label{prop:4} $v_i$ is pure middle vertex iff whenever $a_{ij}=1,\; \exists\; k\; \text{such that}\; a_{kj}=2\; \text{and}\; a_{ki}=1$
 \item  \label{prop:5} $v_{i}$ is a middle end vertex iff it satisfies exactly one of the following
 \begin{itemize}
 \item[a.] $a_{ij}=\frac{1}{4},$ for some $j$
 \item[b.] $a_{ij}=\frac{1}{2},$ and $a_{ik}=2,$ with $a_{jk}=1,$ for some $j$
 \item[c.] $a_{ij}=\frac{1}{2},$ and $\nexists \;k$ such that $a_{ki}=2\;and\; a_{kj}=1$.
 \end{itemize}

\item  \label{prop:6} $v_i$ is a pure end vertex iff it is neither pure middle nor middle end vertex.
 \end{enumerate}

  Let $A$ be a matrix, it is said to be semigraphical if there exist a semigraph whose adjacency matrix is $A$. \\
We give the necessary and sufficient conditions for a matrix to be semigraphical. In \cite{cmd}, authors obtained the necessary and sufficient for their adjaceny matrix being semigraphical. Here, we take a similar approach. 

 \begin{theorem} \label{thm:1}
 A square $n\times n$ matrix $A=(a_{ij})$ is semigraphical if and only if it satisfies the following conditions:
 \begin{enumerate}
 \item  $a_{ij}\in \{0, \frac{1}{4}, \frac{1}{2}, 1, 2, \cdots, n-1\},$  $a_{ij}=a_{ji},\; \forall\; i, j$ and $a_{ii}=0, \forall\; i.$
\item  Let $X=\{1,2, \cdots, n \}$ be an indexing set of columns and rows of $A$. Let $e_1, e_2, \cdots, e_m$ be subsets of X such that $X=\displaystyle \cup_{i=1}^{m}e_i$, where $\forall\; i,\;e_{i} \geq 2 $ and $|e_i \cap e_j|\leq1;\; \forall \;i\neq j$. Also, the submatrix $A_i$ of $A$ with the indexing set $e_i$ is a matrix described as in property 1(upto the permutations of indices).
\end{enumerate}
\end{theorem} 

  \begin{proof}
 Suppose $A$ is a semigraphical matrix. Then the conditions {\it(1) }and {\it(2)} follow from the definition \ref{def:4} and the property \ref{prop:1}.
 \par Conversely, if $A$ satisfies the above conditions then we show the existence of semigraph. We define the vertex set $X=\{1,2,3,\cdots,n\}$ and edge set $E=\{e_1, e_2,\cdots,e_m\}$, where $e_i$ is an indexing set of the submatrix $A_i$. Thus, by the second condition $|e_i|\geq 2,\;\forall i$ and $|e_i \cap e_j|\leq1;\; \forall \;i\neq j$; hence $G=(X, E)$ is a semigraph. 
 
 \end{proof}
 
We give an algorithm to construct semigraph from a semigraphical matrix. 
 \subsection{Algorithm to construct Semigraphs from the semigraphical matrices}
  Let $A$ be a semigraphical matrix with $A=A^{S} +A^{E}$ satisfying the conditions of Theorem \ref{thm:1}
 \begin{enumerate} 
\item[Step 1:] Use $A^{S}$ to lay down the graph skeleton 
\item[Step 2:] for $i\in\{1,2,\cdots,n\}$,\\
 set $N(i)=\{{i}_{j}\in X\; |\; a^{S}_{i i_j}=1,\; \forall \;1\leq j\leq k_i\}$, where $k_i$ is graph degree of $i$.
\item[Step 3:] set $X_{Pend}=\{ p_1,p_2,\cdots, p_r\}$ the set of pure end vertices, \\$X_{Mend}=\{ q_1,q_2,\cdots, q_l\}$ the set of middle end vertices. \\
We create these sets using the properties \ref{prop:4}, \ref{prop:5} and \ref{prop:6}.

\item[Step 4:] Create edges starting with pure end vertices\\
For each $p_{i}\in X_{Pend}$\\ 
For each $j\in N(p_i)$ in $G^{S}$:\\
if $a_{p_{i}j}=\frac{1}{2}$ then set $(p_i, j)$ as an edge and stop\\
elif $a_{p_{i}j}=1$, search for a sequence of column indices $j_{1}, j_{2},\cdots, j_{r}$ such that $p_{i}^{th}$ row has $[1:r]$ in the respective column indices with $j_{r}^{th}$ row having a sequence $[r:1]\; \text{or} \; [r:2, \;\frac{1}{2}]$ in the $p_{i}, j_{1}, j_2,\cdots, j_{r-1}$ columns.  \\
set $e_{i}=(p_{i}, j_1,j_2, \cdots, j_{r})$.
\item[Step 5:] Create edges starting with middle end vertices\\
For each $q_{i}\in X_{Mend}$\\ 
For each $j\in N(q_{i})$ in $G^{S}$:\\
if $a_{p_{i}j}=\frac{1}{4}$ then set $(p_i, j)$ as an (quarter)edge and stop\\
elif $a_{p_{i}j}=\frac{1}{2}$, search for a sequence of column indices $j_{1}, j_{2},\cdots, j_{r}$ such that $p_{i}^{th}$ row has $[\frac{1}{2},2:r]$ in the respective column indices with $j_{r}^{th}$ row having a sequence $[r:1]\; \text{or} \; [r:2, \;\frac{1}{2}]$ in the $p_{i}, j_{1}, j_2,\cdots, j_{r-1}$ columns.  \\
set $e_{i}=(p_{i}, j_1,j_2, \cdots, j_{r})$.

\item[Step 6:] Each edge constructed in step 4 or 5 appears twice, set $E$ as the set of all edges created in step 4 and 5. 
 \end{enumerate}
There is a unique sequence in step 4 and 5(by definition of adjacency and semigraph). Hence, edges created are unique thus the semigraph is unique.
  In step 5, $a_{p_{i}j}=1$ can't be a case as $p_{i}$ is a middle end vertex.

 \section{Bounds on eigenvalues} 
  Let $A$ denote the adjacency matrix of a Semigraph $G$. As the matrix is symmetric, its eigenvalues are real. Let $\lambda_1 \geq \lambda_2 \geq \cdots \geq \lambda_n$ be the eigenvalues of $A.$ We say that $\lambda_1 \geq \lambda_2 \geq \cdots \geq \lambda_n$ are eigenvalues of $G$.  In this section, we obtain some bounds on the largest eigenvalues. For graph theory results please refer to \cite{rbp}. In \cite{cmd}, authors were able to obtain similar results for their adjacency matrix but their adjacency matrix is not symmetric, hence could give uniform treatments to all semigraphs. 
   
  We use $j \sim_{S} i$ to represent $i^{th}$ and $j^{th}$ vertices form a partial edge, 
  $j \sim_{|-} i$ represent $i^{th}$ and $j^{th}$ vertices form a partial half edge, $j \sim_{|-|} i$ represent $i^{th}$ and $j^{th}$ vertices form a quarter edge, $j \sim_{l} i$ represent $i^{th}$ and $j^{th}$ vertices belong to the same edge and are $l$ distance apart.

  \begin{theorem}
  Let $G$ be a Semigraph, $G^{S}$ is the skeleton graph of $G$, $r$ is the rank(size of the largest edge), $\Delta^S$ be the maximum degree in $G^{S}$ then $$ \lambda_{1}\leq \frac{r(r-1)}{2}\Delta^S $$  
  Also, if $\delta$ is the minimum degree of $G$ then $$\delta \leq \lambda_{1}$$
  \end{theorem}
  
  \begin{proof}
  Let $X$ be an eigenvector of the largest eigenvalue $\lambda_1$. Then, for the $i^{th}$ component $x_{i}$ of $X,$ we get 
  $$\lambda x_{i}=\sum_{j \sim_{s} i}x_{j} +\frac{1}{2}\sum_{j \sim_{|-} i}x_{j} +\frac{1}{4}\sum_{j \sim_{|-|} i}x_{j} +\sum_{l=2}^{r-1}\sum_{j \sim_{l} i}lx_{j}  $$ 
 Consider the maximum component $x_{k}>0$ of $X$. 
 \begin{align*}
 \lambda x_{k}=&\sum_{j \sim_{s} k}x_{j} +\frac{1}{2}\sum_{j \sim_{|-} k}x_{j} +\frac{1}{4}\sum_{j \sim_{|-|} k}x_{j} +\sum_{l=2}^{r-1}\sum_{j \sim_{l} k}lx_{j} \\
 \lambda x_{k}\leq&\sum_{j \sim_{s} k}x_{j} +\sum_{j \sim_{|-} k}x_{j} +\sum_{j \sim_{|-|} k}x_{j} +\sum_{l=2}^{r-1}l\left(\sum_{j \sim_{l} k}x_{j}\right)\\
\leq & \;\Delta^S x_{k} +\sum_{l=2}^{r-1}l\Delta^S x_{k} \\
     \leq &\; \Delta^S x_{k} +\Delta^S x_{k} \left(\frac{r(r-1)}{2}-1\right)  
  \end{align*}
  \begin{align*}
\leq & \;\left(\Delta^S +\frac{\Delta^S (r^2-r-2)}{2}\right)x_{k}\\
 \leq &\;\Delta^S\left(1+\frac{ r^2-r-2}{2}\right)x_{k}\\
 \leq &\;\Delta^S\left(\frac{ r^2-r}{2}\right)x_{k}
 \end{align*}
 Thus, $$ \lambda_{1}\leq \frac{r(r-1)}{2}\Delta^S $$
 For second part, we know that $\lambda_1=\underset{||X||=1}{max}X^{t}A{X} =\underset{X\neq 0}{max}\frac{X^{t}A{X}}{X^{t}X}$.
 Thus, $$\lambda_1\geq \frac{\mathbf{1}^{t}A{\mathbf{1}}}{\mathbf{1}^{t}\mathbf{1}}=\frac{\displaystyle \sum_{i=1}^{n}d(v_{i})}{n}\geq \frac{n\delta}{n}$$
 Hence, $\delta \leq \lambda_{1}$. 
  \end{proof}

We note that when Semigraph is graph, we get the graph theory result \cite[Theorem 3.10]{rbp} as the consequence of above result.

  To prove the next result, we look at the expression for degree sum in semigraph. If a semigraph has $m$ edges then we decompose $m$ as $m=m_1+m_2+m_3+m_4$ where $m_1$ is the number of full edges, $m_2$ is the number of quarter edges, $m_3$ is the number of half edges with one partial half edge and $m_4$ is the number of half edges with two partial half edges.
  \begin{lemma} \label{lemma:1}
  Let $E=\{e_1,e_2,\cdots, e_m\}$ be an edge set of semigraph $G$ on $n$ vertices such that $|e_i|=r_i$, $d_1, d_2,\cdots, d_n$ are the degrees, then $$\displaystyle \sum_{i=1}^{n}d_i=\frac{1}{3}\sum_{i=1}^{m}r_i(r^2_i -1)-\frac{3}{2}m_2-\frac{1}{2}m_3-m_4.$$
  \end{lemma}
  \begin{proof}
Consider an edge $e_i=(v_{i_1}, v_{i_2}, \cdots, v_{i_{r_i}})$, assume that $v_{i_1}, v_{i_{r_i}}$ are pure end vertices. The degree of the $i_j^{th}$ vertex with respect to the edge $e_i$ is given by 
  \begin{align*}
  &(j-1)+\cdots+2+1+1+2+\cdots+(r_i -j)\\
  =&\frac{j(j-1)}{2}+\frac{(r_i -j)(r_i -j+1)}{2}\\
  =&\frac{1}{2}[2j^2-2j-2jr_i+r_i(r_i+1)]\\
  =& j^2-j-jr_i+\frac{r_i(r_i+1)}{2}
   \end{align*}
  Thus, the degree contribution of this edge to the total degree sum is 
  $$\sum_{j=1}^{r_i} \left [j^2-j-jr_i+\frac{r_i(r_i+1)}{2}\right ] $$
  Simplifying this, we get
  \begin{align*}
  &=\sum_{j=1}^{r_i} j^2- \sum_{j=1}^{r_i} j-r_i \sum_{j=1}^{r_i} j+\frac{r_i(r_i+1)}{2} \sum_{j=1}^{r_i} 1\\
  &=\frac{r_i(r_i+1)(2r_i+1)}{6}-\frac{r_i(r_i+1)}{2}-\frac{r^2_i(r_i+1)}{2}+\frac{r^2_i(r_i+1)}{2} \\
  &=r_i(r_i+1)\left[ \frac{1}{6}(2r_i+1)-\frac{1}{2}\right]\\
  &=r_i(r_i+1)\left[ \frac{r_i}{3}-\frac{1}{3}\right]\\
  &=\frac{1}{3}r_i(r^2_i-1)
  \end{align*}
  Thus, total degree contribution of all edges having both end vertices as pure end vertices is $\displaystyle \frac{1}{3}\sum_{i=1}^{m_1}r_i(r^2_i -1)$. \\
  For a quarter edge $(u, v)$, its degree contribution to the total degree is $\frac{1}{4}+\frac{1}{4}$. Thus, total degree contribution to the total degree sum due to quarter edges is $\frac{1}{2}m_2$, which can be written as $\displaystyle \frac{1}{3}\sum_{i=1}^{m_2}2(2^2 -1)-\frac{3}{2}m_2$. \\
 If one of the end vertices of an edge say $v_1$ is a middle end vertex, then the total degree contribution of that edge to the total degree would be $\frac{1}{3}r_i(r^2_i-1)-\frac{1}{2}$ as the degree contribution due to $v_{i_1}$ is $\frac{1}{2}+2+\cdots+(r_i -1)$. Thus, total degree contribution of all edges having one of end vertices as middle end vertex is $$\displaystyle \frac{1}{3}\sum_{i=1}^{m_3}r_i(r^2_i -1)-\frac{1}{2}m_3$$
If both end vertices of an edge are middle end vertices, then the total degree contribution of that edge to the total degree would be $\frac{1}{3}r_i(r^2_i-1)-\frac{1}{2}-\frac{1}{2}$ as the degree contribution due to $v_{i_1}$ and $v_{i_{r_i}}$ is $\frac{1}{2}+2+\cdots+(r_i -1)$ for each. Thus, total degree contribution of all edges having both end vertices as middle end vertices is $$\displaystyle \frac{1}{3}\sum_{i=1}^{m_4}r_i(r^2_i -1)-m_4$$
Thus, the degree sum $$\displaystyle \sum_{i=1}^{n}d_i=\frac{1}{3}\sum_{i=1}^{m}r_i(r^2_i -1)-\frac{3}{2}m_2-\frac{1}{2}m_3-m_4$$
   \end{proof}
  
  When $G$ is a graph then we can see that $r_i=2,\; \forall i$ and $m_2=m_3=m_4=0$, thus we get
   $$ \displaystyle \sum_{i=1}^{n}d_i=\frac{1}{3}\sum_{i=1}^{m}2(2^2-1)=2m.$$

    \begin{lemma} \label{lemma:2}
  Let $E=\{e_1,e_2,\cdots, e_m\}$ be an edge set of semigraph $G$ on $n$ vertices such that $|e_i|=r_i$, $\lambda_1, \lambda_2, \cdots, \lambda_n$ are the eigenvalues of $G$,  then $$\displaystyle \sum_{i=1}^{n}\lambda^2_i=\frac{1}{6}\sum_{i=1}^{m}r^2_i(r^2_i -1)-\frac{15}{8}m_2-\frac{3}{4}m_3-\frac{1}{2}m_4.$$
  \end{lemma}
  \begin{proof}
We know that if $\lambda_1, \lambda_2, \cdots, \lambda_n$ are eigenvalues of $A$ then $\lambda^2_1, \lambda^2_2, \cdots, \lambda^2_n$ are eigenvalues of $A^2$. And trace of $A^2$ is $\displaystyle \sum_{i=1}^{n}\lambda^2_i$. \\
Consider an edge $e_i=(v_{i_1}, v_{i_2}, \cdots, v_{i_{r_i}})$, assume that $v_{i_1}, v_{i_{r_i}}$ are pure end vertices. The contribution of the $i_j^{th}$ vertex with respect to the edge $e_i$ to the trace of $A^2$ is given by 
  \begin{align*}
  &(j-1)^2+\cdots+2^2+1^2+1^2+2^2+\cdots+(r_i -j)^2\\
  =&\frac{j(j-1)(2j-1)}{6}+\frac{(r_i -j)(r_i -j+1)(2r_i-2j+1)}{6}\\
  =&\frac{1}{6}[2j^3-3j^2+j+2r^3_i+3r^2_i-6r^2_ij+6r_ij^2-6r_ij+r_i-2j^3+3j^2-j]\\
=&\frac{1}{6}\left [2r^3_i+3r^2_i+r_i-6r^2_ij+6r_ij^2-6r_ij\right ]\
   \end{align*}
  Thus, the contribution of this edge to the trace of $A^2$ is 
  $$\frac{1}{6}\sum_{j=1}^{r_i} \left [ 2r^3_i+3r^2_i+r_i-6r^2_ij+6r_ij^2-6r_ij \right ] $$
  
  Simplifying this, we get
  \begin{align*}
  &=\frac{1}{6}\left [2\sum_{j=1}^{r_i}r^3_i +3\sum_{j=1}^{r_i}r^2_i+\sum_{j=1}^{r_i}r_i-6r^2_i \sum_{j=1}^{r_i} j+6r_i \sum_{j=1}^{r_i} j^2-6r_i\sum_{j=1}^{r_i}j\right] \\
  &=\frac{1}{6}\left [2r^4_i+3r^3_i+r^2_i-6r^2_i\frac{r_i(r_i+1)}{2}+6r_i\frac{r_i(r_i+1)(2r_i+1)}{6}-6r_i\frac{r_i(r_i+1)}{2}\right] \\
  &=\frac{1}{6}\left[r^4_i-r^2_i\right]\\
  &=\frac{1}{6}r^2_i(r^2_i-1)
  \end{align*}
  Thus, the total contribution of all edges having both end vertices as pure end vertices to the trace of $A^2$ is $\displaystyle \frac{1}{6}\sum_{i=1}^{m_1}r^2_i(r^2_i -1)$. \\
  For a quarter edge $(u, v)$, its contribution to the trace of $A^2$ is $\frac{1}{16}+\frac{1}{16}$. Thus, total contribution to the trace of $A^2$ due to quarter edges is $\frac{1}{8}m_2$, which can be written as $\displaystyle \frac{1}{6}\sum_{i=1}^{m_2}2^2(2^2 -1)-\frac{15}{8}m_2$. \\
 If only one of the end vertices of an edge say $v_{i_1}$ is a middle end vertex, then the total contribution of that edge to the the trace of $A^2$ would be $\frac{1}{6}r^2_i(r^2_i-1)-\frac{1}{4}$ as the contribution due to $v_{i_1}$ is $\frac{1}{4}+2^2+\cdots+(r_i -1)^2$. Thus, total contribution of all edges having one of end vertices as middle end vertex is $$\displaystyle \frac{1}{6}\sum_{i=1}^{m_3}r^2_i(r^2_i -1)-\frac{3}{4}m_3$$
If both end vertices of an edge are middle end vertices, then the total contribution of that edge to the trace of $A^2$ would be $\frac{1}{6}r^2_i(r^2_i-1)-\frac{1}{4}-\frac{1}{4}$ as the contribution due to $v_{i_1}$ and $v_{i_{r_i}}$ is $\frac{1}{4}+2^2+\cdots+(r_i -1)^2$ for each. Thus, total contribution of all edges having both end vertices as middle end vertices is $$\displaystyle \frac{1}{6}\sum_{i=1}^{m_4}r^2_i(r^2_i -1)-\frac{1}{2}m_4$$
Thus, the trace of $A^2$ is $$\displaystyle \sum_{i=1}^{n}\lambda^2_i=\frac{1}{6}\sum_{i=1}^{m}r^2_i(r^2_i -1)-\frac{15}{8}m_2-\frac{3}{4}m_3-\frac{1}{2}m_4$$
   \end{proof}

  When $G$ is a graph then we can see that $r_i=2,\; \forall i$ and $m_2=m_3=m_4=0$, thus we get
   $$ \displaystyle \sum_{i=1}^{n}d_i=\sum_{i=1}^{n}\lambda^2_i=\frac{1}{6}\sum_{i=1}^{m}2^2(2^2-1)=2m.$$

  \begin{theorem}
Let $G=(X, E)$ be a semigraph with $|X|=n,\; |E|=m$, and let $\lambda_1\geq \lambda_2\geq \cdots \geq\lambda_n $ be the eigenvalues of $G$ then $$ \displaystyle   \lambda_1 \leq \sqrt{ \left(\frac{1}{6}\sum_{i=1}^{m}r^2_i(r^2_i -1)-\frac{15}{8}m_2-\frac{3}{4}m_3-\frac{1}{2}m_4\right)\left(\frac{n-1}{n}\right)}$$
  \end{theorem}
  \begin{proof}
  Notice that $\displaystyle \sum_{i=1}^{n} \lambda_{i}=0$, implies $\displaystyle \lambda_{1}\leq  \sum_{i=2}^{n} |\lambda_{i}|.$\\
 By lemma \ref{lemma:2}, we have 
  
  \begin{align*}
 \frac{1}{6}\sum_{i=1}^{m}r^2_i(r^2_i -1)-\frac{15}{8}m_2-\frac{3}{4}m_3-\frac{1}{2}m_4 &= \sum_{i=1}^{n}\lambda^2_i\\
\frac{1}{6}\sum_{i=1}^{m}r^2_i(r^2_i -1)-\frac{15}{8}m_2-\frac{3}{4}m_3-\frac{1}{2}m_4 -\lambda^2_1 &=\sum_{i=2}^{n}\lambda^2_i
  \end{align*}
  By Cauchy-Schwarz inequality and $\displaystyle \lambda_{1}\leq  \sum_{i=2}^{n} |\lambda_{i}|$, we get 
 $$ \displaystyle \frac{1}{6}\sum_{i=1}^{m}r^2_i(r^2_i -1)-\frac{15}{8}m_2-\frac{3}{4}m_3-\frac{1}{2}m_4 -\lambda^{2}_{1} \geq \frac{1}{n-1} \left (\sum_{i=2}^{n} |\lambda_{i}|\right)^2\geq \frac{\lambda^{2}_{1}}{n-1} $$
 Simplifying it further, we get
    \begin{align*}
   \displaystyle \frac{1}{6}\sum_{i=1}^{m}r^2_i(r^2_i -1)-\frac{15}{8}m_2-\frac{3}{4}m_3-\frac{1}{2}m_4 -\lambda^{2}_{1}& \geq \frac{\lambda^{2}_{1}}{n-1} \\
    \displaystyle \frac{1}{6}\sum_{i=1}^{m}r^2_i(r^2_i -1)-\frac{15}{8}m_2-\frac{3}{4}m_3-\frac{1}{2}m_4 &\geq \lambda^{2}_1 \left( 1+\frac{1}{n-1}\right)\\
    \displaystyle \left(\frac{1}{6}\sum_{i=1}^{m}r^2_i(r^2_i -1)-\frac{15}{8}m_2-\frac{3}{4}m_3-\frac{1}{2}m_4\right) \left( \frac{n-1}{n}\right)&\geq \lambda^{2}_1 
\end{align*}
Thus,
 \begin{align*}
 \displaystyle   \lambda^{2}_1 \leq & \left(\frac{1}{6}\sum_{i=1}^{m}r^2_i(r^2_i -1)-\frac{15}{8}m_2-\frac{3}{4}m_3-\frac{1}{2}m_4\right)\left(\frac{n-1}{n}\right)\\
 \displaystyle   \lambda_1 \leq &\sqrt{ \left(\frac{1}{6}\sum_{i=1}^{m}r^2_i(r^2_i -1)-\frac{15}{8}m_2-\frac{3}{4}m_3-\frac{1}{2}m_4\right)\left(\frac{n-1}{n}\right)}
\end{align*}
  \end{proof}

When $G$ is a graph, $r_i=2,\;\forall\; i$ and $m_2=m_3=m_4=0$, we get 
\begin{align*}
 \displaystyle   \lambda_1 \leq &\sqrt{ \left(\frac{1}{6}\sum_{i=1}^{m}2^2(2^2 -1)-0\right)\left(\frac{n-1}{n}\right)}\\
  \displaystyle   \lambda_1 \leq &\sqrt{ 2m\left(\frac{n-1}{n}\right)}
 \end{align*}

\section{Spectra of Star Semigraphs}
  In this section, we study the spectra of two types of Star semigraphs.
\begin{definition}
If $\lambda_1,\lambda_2,\cdots, \lambda_n$ are eigenvalues of a matrix $A$ with multiplicities $k_1, k_2, \cdots, k_n$ then spectrum of $A$ is
$$\begin{pmatrix}\lambda_1& \lambda_2&\cdots& \lambda_n \\ k_1&k_2&\cdots&k_n\end{pmatrix}$$ 
\end{definition}  
  \subsection{ Spectra of Star semigraphs- type I}
 Let $S^3_{2, n}$ denote a star semigraph having one edge of  3 vertices and n edges of 2 vertices incident on the middle vertex of the edge of 3 vertices. 
\begin{figure}[h]
\centering
  \begin{tikzpicture}[yscale=0.5]
 \Vertex[size=0.2, y=3, label=$v_2$, position=above, color=black]{B} 
 \Vertex[size=0.2,  label=$v_1$, position=180, color=none]{A} 
 \Vertex[size=0.2,y=-3,  label=$v_3$, position=below, color=black]{C} 
 \Edge(A)(B) \Edge(A)(C)
  \Vertex[size=0.2,x=1.1, y=2.5,  label=$v_4$, position=45, color=black]{D} 
    \Vertex[size=0.2,x=1.3, y=-3, label=$v_r$, position=right, color=black]{E} 
  \Vertex[size=0.2,x=-1.6, y=2.2, label=$v_{n+3}$, position=left, color=black]{F} 
  \Vertex[size=0.2,x=-1.2, y=-3.15, label=$v_{r+1}$, position=left, color=black]{G} 
  \draw[thick](0.2,0.2)--(1.1,2.45);
    \draw[thick](0.15,-0.3)--(1.15,-2.7);
    \draw[thick](-0.2,0.3)--(-1.5, 2.1);
    \draw[thick](-0.2,-0.3)--(-1.2,-3.1);
    \draw[thick](0.2, 0)--(0.2, 0.5);
   \draw[thick](0.15, -0.15)--(0.15, -0.6);
    \draw[thick](-0.2, 0)--(-0.2, 0.5);
   \draw[thick](-0.2, -0.1)--(-0.15, -0.6);
  
  \Edge[bend =20, style={dashed}](D)(E)
   \Edge[bend =20, style={dashed}](G)(F)
 
 \Vertex[size=0.2, x=4,y=3, label=$v_2$, position=above, color=black]{B} 
 \Vertex[size=0.2, x=4, label=$v_1$, position=left, color=none]{A} 
 \Vertex[size=0.2,x=4,y=-3,  label=$v_3$, position=below, color=black]{C} 
 \Edge(A)(B) \Edge(A)(C)

\end{tikzpicture}\\
$\text{Star semigraphs:}\;\;S^3_{2,n}  \hspace{2.5cm}S^3_{2,0}$
\caption{ }
\label{fig:8}
\end{figure}

 The adjacency matrix $A$ is as follows: 
  \[
     \bordermatrix{ & {v_1} & {v_2} & {v_3}& {v_4} & {v_5} & \cdots & {v_{n+3}} \cr
       v_1 & 0&1&1&\frac{1}{2}&\frac{1}{2}&\cdots&\frac{1}{2} \cr
       v_2 & 1&0&2&0&0&\cdots&0\cr
       v_3 & 1&2&0&0&0&\cdots&0 \cr
        v_4 & \frac{1}{2}&0&0&0&0&\cdots&0 \cr
       v_5 & \frac{1}{2}&0&0&0&0&\cdots&0 \cr
       \vdots & \vdots&\vdots&\vdots&\vdots&\vdots&\ddots&\vdots\cr
       v_{n+3} &\frac{1}{2}&0&0&0&0&\cdots&0} \qquad 
 \]

\begin{lemma}
The characteristic polynomial of adjacency matrix of $S^3_{2, n}$  is 
$$P_{n}(\lambda)=\displaystyle \lambda^{n-1}(\lambda+2)\left(\lambda^3 -2\lambda^2-\frac{n+8}{4}\;\lambda+\frac{n}{2}\right)$$
\end{lemma}
\begin{proof}
Consider the characteristic polynomial $P_{n}(\lambda) =det(\lambda I -A)$
$$P_{n}(\lambda)=\begin{vmatrix}
\lambda&-1&-1&-\frac{1}{2}&\cdots&-\frac{1}{2}&-\frac{1}{2} \\
-1&\lambda&-2&0&\cdots&0&0 \\
-1&-2&\lambda&0&\cdots&0&0\\
-\frac{1}{2}&0&0&\lambda&\cdots&0&0 \\
 \vdots&\vdots&\vdots&\vdots&\ddots&\vdots&\vdots  \\ 
 -\frac{1}{2}&0&0&0&\cdots&\lambda&0\\
 -\frac{1}{2}&0&0&0&\cdots&0&\lambda
\end{vmatrix}_{(n+3)\times (n+3)} $$

Using cofactor expansion along the last column, we get
$$=(-1)^{n+4}\frac{(-1)}{2}\begin{vmatrix}
-1&\lambda&-2&0&\cdots&0 \\
-1&-2&\lambda&0&\cdots&0\\
-\frac{1}{2}&0&0&\lambda&\cdots&0 \\
 \vdots&\vdots&\vdots&\vdots&\ddots&\vdots  \\ 
 -\frac{1}{2}&0&0&0&\cdots&\lambda\\
 -\frac{1}{2}&0&0&0&\cdots&0
\end{vmatrix}  
+(-1)^{2n+6}\lambda \begin{vmatrix}
\lambda&-1&-1&-\frac{1}{2}&\cdots&-\frac{1}{2} \\
-1&\lambda&-2&0&\cdots&0 \\
-1&-2&\lambda&0&\cdots&0\\
-\frac{1}{2}&0&0&\lambda&\cdots&0 \\
 \vdots&\vdots&\vdots&\vdots&\ddots&\vdots  \\ 
 -\frac{1}{2}&0&0&0&\cdots&\lambda
\end{vmatrix} 
$$
Observe that the second term is the characteristic polynomial of $S^3_{2, n-1}$. \\
Thus, we get
$$=(-1)^{n}\frac{(-1)}{2}\begin{vmatrix}
-1&\lambda&-2&0&\cdots&0 \\
-1&-2&\lambda&0&\cdots&0\\
-\frac{1}{2}&0&0&\lambda&\cdots&0 \\
 \vdots&\vdots&\vdots&\vdots&\ddots&\vdots  \\ 
 -\frac{1}{2}&0&0&0&\cdots&\lambda\\
 -\frac{1}{2}&0&0&0&\cdots&0
\end{vmatrix}_{(n+2)\times (n+2)}  
+\lambda P_{n-1}(\lambda)
$$
By using the cofactor expansion for the first term along the last row, we get
$$=(-1)^{n}(-1)^{n+3}\frac{(-1)}{2}\frac{(-1)}{2}\begin{vmatrix}
\lambda&-2&0&\cdots&0 \\
-2&\lambda&0&\cdots&0\\
0&0&\lambda&\cdots&0 \\
 \vdots&\vdots&\vdots&\ddots&\vdots  \\ 
 0&0&0&\cdots&\lambda\\
\end{vmatrix}
+\lambda P_{n-1}(\lambda)$$
The determinant of the above $(n+1)\times (n+1)$ matrix is $\left(\lambda^2 -4\right)\lambda^{n-1} .$ 
\begin{align*}
P_{n}(\lambda)=&-\frac{1}{4}\left(\lambda^2 -4\right)\lambda^{n-1}+\lambda P_{n-1}(\lambda) \\
=& -\frac{1}{4}\left(\lambda^2 -4\right)\lambda^{n-1} +\lambda \left [ -\frac{1}{4}\left(\lambda^2 -4\right)\lambda^{n-2}+\lambda P_{n-2}(\lambda)\right ] \\
=&-\frac{1}{4}\lambda^{n-1} \left(\lambda^2 -4\right)(2)+\lambda^2 P_{n-2}(\lambda)\\
=&-\frac{1}{4}\lambda^{n-1} \left(\lambda^2 -4\right)(2)+\lambda^2 \left [ -\frac{1}{4}\left(\lambda^2 -4\right)\lambda^{n-3}+\lambda P_{n-3}(\lambda)\right ] 
\end{align*}

\begin{align*}
=&-\frac{1}{4}\lambda^{n-1} \left(\lambda^2 -4\right)(3)+\lambda^3 P_{n-3}(\lambda)\\
&\vdots \\
=& -\frac{1}{4}\lambda^{n-1} \left(\lambda^2 -4\right)(n)+\lambda^n P_{0}(\lambda)
\end{align*}
Here, $ P_{0}(\lambda)$ is the characteristic polynomial of $S^{3}_{2, 0}$, Fig. \ref{fig:8}. 

\begin{align*} 
P_0(\lambda)=&\begin{vmatrix}\lambda &-1&-1\\ -1&\lambda &-2 \\ -1&-2&\lambda \end{vmatrix} \\
P_0(\lambda)=&\lambda^3-6\lambda -4
\end{align*}
Thus, 
\begin{align*}
P_{n}(\lambda )= &-\frac{1}{4}\lambda^{n-1} \left(\lambda^2 -4\right)(n)+\lambda^n(\lambda^3-6\lambda -4) \\
P_{n}(\lambda )= &\lambda^{n+3} -\frac{n+24}{4}\;\lambda^{n+1}-4\lambda^n+n\lambda^{n-1} \\
P_{n}(\lambda )= & \lambda^{n-1}\left(\lambda^4 -\frac{n+24}{4}\;\lambda^2-4\lambda +n\right)\\
P_{n}(\lambda)=&\displaystyle \lambda^{n-1}(\lambda+2)\left(\lambda^3 -2\lambda^2-\frac{n+8}{4}\;\lambda+\frac{n}{2}\right)
\end{align*}
\end{proof}
Thus, the spectra of star semigraphs are:$$\begin{pmatrix}0&-2&\lambda_1& \lambda_2& \lambda_3 \\ n-1&1&1&1&1\end{pmatrix}$$ 
 where $\lambda_{1}, \lambda_{2},\lambda_{3}$ are roots of the cubic polynomial $\lambda^3 -2\lambda^2-\frac{n+8}{4}\;+\frac{n}{2}.$

  \subsection{ Spectra of 3-uniform Star semigraphs- type II }
 Let $S^3_{ n}$ denote a star semigraph having all edges with  3 vertices and having middle vertex as the common vertex to all. This is special type of 3-uniform semigraph.
\begin{figure}[h]
\centering
  \begin{tikzpicture}[yscale=0.5]
 \Vertex[size=0.2, y=3, label=$v_1$, position=above, color=black]{B} 
 \Vertex[size=0.2,  label=$v_0$, position=left, color=none]{A} 
 \Vertex[size=0.2,y=-3,  label=$v_2$, position=below, color=black]{C} 
 \Edge(A)(B) \Edge(A)(C)
  \Vertex[size=0.2,x=1.1, y=2.5,  label=$v_3$, position=45, color=black]{D} 
    \Vertex[size=0.2,x=1.3, y=-3, label=$v_{2n-1}$, position=right, color=black]{E} 
  \Vertex[size=0.2,x=-1.6, y=2.2, label=$v_{2n}$, position=left, color=black]{F} 
  \Vertex[size=0.2,x=-1.2, y=-3.15, label=$v_{4}$, position=left, color=black]{G} 
  \Edge(A)(D) \Edge(A)(E) \Edge(A)(F) \Edge(A)(G)
  
  \Edge[bend =20, style={dashed}](D)(E)
   \Edge[bend =20, style={dashed}](G)(F)
\end{tikzpicture}\\
$S^3_{n} $ 
\caption{}
\label{fig:9}
\end{figure}
 The adjacency matrix $A$ is as follows: 
  \[
     \bordermatrix{ & {v_0} & {v_1} & {v_2}& {v_3} & {v_4} & \cdots & v_{2n-3}&v_{2n-2}&v_{2n-1}&v_{2n} \cr
       v_0 & 0&1&1&1&1&\cdots&1&1&1&1 \cr
       v_1 & 1&0&2&0&0&\cdots&0&0&0&0\cr
       v_2 & 1&2&0&0&0&\cdots&0&0&0&0 \cr
        v_3 & 1&0&0&0&2&\cdots&0&0&0&0 \cr
       v_4 & 1&0&0&2&0&\cdots&0&0&0&0 \cr
       \vdots & \vdots&\vdots&\vdots&\vdots&\vdots&\ddots&\vdots\cr
       v_{2n-3} &1&0&0&0&0&\cdots&0&2&0&0\cr
         v_{2n-2} &1&0&0&0&0&\cdots&2&0&0&0\cr
         v_{2n-1} &1&0&0&0&0&\cdots&0&0&0&2\cr
           v_{2n} &1&0&0&0&0&\cdots&0&0&2&0 } \qquad 
 \]

\begin{lemma}
The characteristic polynomial of adjacency matrix of $S^3_{n}$  is 
$$P_{n}(\lambda)=\displaystyle (\lambda+2)(\lambda^2-4)^{n-1}\left[ \lambda^2 -2\lambda-2n\right ]$$
\end{lemma}
\begin{proof}
Consider the characteristic polynomial $S^3_{n}(\lambda) =det(\lambda I -A)$
$$P_{n}(\lambda)=\begin{vmatrix}
\lambda&-1&-1&-1&-1&\cdots&-1&-1&-1&-1 \\
-1&\lambda&-2&0&0&\cdots&0&0&0&0 \\
-1&-2&\lambda&0&0&\cdots&0&0&0&0\\
-1&0&0&\lambda&-2&\cdots&0&0&0&0 \\
-1&0&0&-2&\lambda&\cdots&0&0&0&0\\
 \vdots&\vdots&\vdots&\vdots&\ddots&\vdots&\vdots  \\ 
 -1&0&0&0&0&\cdots&\lambda&-2&0&0\\
 -1&0&0&0&0&\cdots&-2&\lambda&0&0 \\
 -1&0&0&0&0&\cdots&0&0&\lambda&-2\\
 -1&0&0&0&0&\cdots&0&0&-2&\lambda
\end{vmatrix}_{(2n+1)\times (2n+1)} $$
Using co-factor expansion along the last column, we get
\begin{align*}
P_{n}(\lambda)=&\lambda A_{(2n+1)(2n+1)} +2A_{(2n)(2n+1)} -1A_{1(2n+1)} 
\end{align*}
where $A_{ij}$ is the determinant of the submatrix of $A$ obtained by deleting $i^{th}$ row and $j^{th}$ column.
Further, using the co-factor expansion of $A_{(2n+1)(2n+1)}$ along the last column and evaluating the determinant, we get $$\lambda A_{(2n+1)(2n+1)}=\lambda^2P_{n-1}-\lambda(\lambda^2-4)^{n-1}.$$ Taking the co-factor expansion of $A_{(2n)(2n+1)}$ along the along the last column and simplifying we get $$2A_{(2n)(2n+1)}=-4P_{n-1}-2(\lambda^2-4)^{n-1}.$$ And, taking the co-factor expansion of $A_{1(2n+1)}$ along the last row and simplifying it we get $$-A_{1(2n+1)}=-2(\lambda^2-4)^{n-1}-\lambda(\lambda^2-4)^{n-1}$$
Thus, \begin{align*}
P_{n}(\lambda)=&\lambda^2P_{n-1}-\lambda(\lambda^2-4)^{n-1} -4P_{n-1}-2(\lambda^2-4)^{n-1}-2(\lambda^2-4)^{n-1}-\lambda(\lambda^2-4)^{n-1}\\
=&(\lambda^2-4)P_{n-1}-2(\lambda+2)(\lambda^2-4)^{n-1}
\end{align*}
Applying the above formula for $P_{n-1}$ and simplifying it we get
 \begin{align*}
P_{n}(\lambda)=&(\lambda^2-4)\left[(\lambda^2-4)P_{n-2}-2(\lambda+2)(\lambda^2-4)^{n-2}\right]-2(\lambda+2)(\lambda^2-4)^{n-1}\\
P_{n}(\lambda)=&(\lambda^2-4)^2P_{n-2}-2\times 2(\lambda+2)(\lambda^2-4)^{n-1}
\end{align*}
Continuing this recursively we get
\begin{align*}
P_{n}(\lambda)=&(\lambda^2-4)^{n-1}P_{1}-2(n-1)(\lambda+2)(\lambda^2-4)^{n-1}
\end{align*}
Note that $P_1(\lambda)=\lambda^3-6\lambda-4$. Hence, 
\begin{align*}
P_{n}(\lambda)=&(\lambda^2-4)^{n-1}\left[\lambda^3-6\lambda-4\right]-2(n-1)(\lambda+2)(\lambda^2-4)^{n-1}
\end{align*}
Simplifying this we get
$$P_{n}(\lambda)=\displaystyle (\lambda^2-4)^{n-1}\left[ \lambda^3 -2(n+2)\lambda-4n\right ]$$
We can see that $\lambda+2$ is a factor. 
Thus, we get
$$P_{n}(\lambda)=\displaystyle (\lambda+2)(\lambda^2-4)^{n-1}\left[ \lambda^2 -2\lambda-2n\right ]$$
\end{proof}

Thus, the spectra of star semigraphs are:$$\begin{pmatrix}-2&2& 1-\sqrt{2n+1}& 1+\sqrt{2n+1} \\ n&n-1&1&1\end{pmatrix}$$

\section*{Conclusion}
By defining the adjacency matrix of a semigraph in a unique way such that this matrix is symmetric opens up a pandora's box from which one can choose any problem starting with spectra of various semigraphs to defining/generlizing results similar to graphs. We have just initiated this study which is expected to result in rich theory regarding semigraphs and also generate many applications, especially in the field of upcoming IOT and AI.
  \section*{Acknowledgment} 
I am grateful to Dr. Charusheela Deshpande for introducing me to the definition of semigraphs and for sharing her insights with me during the course of this research article. I also would like to thanks Dr. Bhagyashri Athawale for example \ref{fig:7} and continuous encouragement.

\bibliographystyle{amsplain}

\end{document}